\documentclass[12pt,a4paper]{article}%
\usepackage[utf8]{inputenc}
\usepackage{hyperref}
\usepackage{amsmath}
\usepackage{amsfonts}
\usepackage{amssymb}
\usepackage{graphicx}%
\providecommand{\U}[1]{\protect\rule{.1in}{.1in}}
\newtheorem{theorem}{Theorem}

\newtheorem{conjecture}[theorem]{Conjecture}
\newtheorem{corollary}[theorem]{Corollary}

\newtheorem{lemma}[theorem]{Lemma}

\newtheorem{proposition}[theorem]{Proposition}
\newtheorem{remark}[theorem]{Observation}

\newenvironment{proof}[1][Proof]{\noindent\textbf{#1.} }{\ \hfill \rule{0.5em}{0.5em}\bigskip}
\graphicspath{{D:/Dropbox/Riste-Sedlar/MetricDim/Jelena/Slike/}}

\begin{document}

\title{Extremal mixed metric dimension with respect to the cyclomatic number}
\author{Jelena Sedlar$^{1}$,\\Riste \v Skrekovski$^{2,3}$ \\[0.3cm] {\small $^{1}$ \textit{University of Split, Faculty of civil
engineering, architecture and geodesy, Croatia}}\\[0.1cm] {\small $^{2}$ \textit{University of Ljubljana, FMF, 1000 Ljubljana,
Slovenia }}\\[0.1cm] {\small $^{3}$ \textit{Faculty of Information Studies, 8000 Novo
Mesto, Slovenia }}\\[0.1cm] }
\maketitle

\begin{abstract}
In a graph $G$, the cardinality of the smallest ordered set of vertices that
distinguishes every element of $V(G)\cup E(G)$ is called the mixed metric
dimension of $G,$ and it is denoted by $\mathrm{mdim}(G).$ In
\cite{SedSkrekMixed}\ it was conjectured that for a graph $G$ with cyclomatic
number $c(G)$ it holds that $\mathrm{mdim}(G)\leq L_{1}(G)+2c(G)$ where
$L_{1}(G)$ is the number of leaves in $G$. It is already proven that the
equality holds for all trees and more generally for graphs with edge-disjoint
cycles in which every cycle has precisely one vertex of degree $\geq3$. In
this paper we determine that for every Theta graph $G,$ the mixed metric
dimension $\mathrm{mdim}(G)$ equals $3$ or $4$, with $4$ being attained if and
only if $G$ is a balanced Theta graph. Thus, for balanced Theta graphs the
above inequality is also tight. We conclude the paper by further conjecturing
that there are no other graphs, besides the ones mentioned here, for which the
equality $\mathrm{mdim}(G)=L_{1}(G)+2c(G)$ holds.

\end{abstract}



\section{Introduction}

Let $G$ be a simple connected graph with $n$ vertices and $m$ edges. The
distance between a pair of vertices $u,v\in V(G)$ is defined as the length of
the shortest path connecting $u$ and $v$ in $G$ and is denoted by $d_{G}%
(u,v)$. The distance between a vertex $u\in V(G)$ and an edge $e=vw\in E(G)$
is defined by $d_{G}(u,e)=d_{G}(u,vw)=\min\{d_{G}(u,v),d_{G}(u,w)\}$. For both
these distances we simply write $d(u,v)$ and $d(u,e)$ if no confusion arises.
We say that a vertex $s\in V(G)$ \emph{distinguishes} (or \emph{resolves}) a
pair $x,x^{\prime}\in V(G)\cup E(G)$ if $d(s,x)\not =d(s,x^{\prime}).$ We say
that a set $S\subseteq V(G)$ is a \emph{mixed metric generator} if every pair
$x,x^{\prime}\in V(G)\cup E(G)$ is distinguished by at least one vertex from
$S$. The cardinality of the smallest mixed metric generator is called the
\emph{mixed metric dimension} of $G,$ and it is denoted by $\mathrm{mdim}(G).$

The notion of the mixed metric dimension is the natural generalization of the
notions of the vertex metric dimension and the edge metric dimension which are
defined as the cardinality of the smallest set of vertices which distinguishes
all pairs of vertices and all pairs of edges respectively. The notion of
vertex metric dimension for graphs was independently introduced by
\cite{Harary1976} and \cite{Slater1975}, under the names resolving sets and
locating sets, respectively. Even before, this notion was introduced for the
realm of metric spaces \cite{Blumenthal1953}. The concept of vertex metric
dimension was recently extended from resolving vertices to resolving edges of
a graph by Kelenc, Tratnik and Yero~\cite{Kel}, which lead to the definition
of the edge metric dimension. Finally, it was further extended to resolving
mixed pairs of edges and vertices by Kelenc, Kuziak, Taranenko, and
Yero~\cite{Kelm} which resulted with the notion of the mixed metric dimension.
All these variations of metric dimensions attracted interest (see
\cite{Peterin2020, SedSkreBounds, SedSkrekMixed, Zhu, Zubrilina}), while for a
wider and systematic introduction of the topic metric dimension that
encapsulates all three above mentioned variations, we recommend the PhD thesis
of Kelenc~\cite{KelPhD}.

In literature, among other questions, the mixed metric dimension of trees,
unicyclic graphs and graphs with edge disjoint cycles was studied. Denoting by
$L_{1}(G)$ the number of leaves in a graph $G,$ we first cite the following
result from \cite{Kelm}.

\begin{proposition}
\label{Prop_Kelm}For every tree $T$, it holds
\[
\mathrm{mdim}(T)=L_{1}(T).
\]

\end{proposition}

A graph in which all cycles are pairwise edge disjoint is called a
\emph{cactus graph}. Having that in mind, the following results were proven in
\cite{SedSkrekMixed}, first for unicyclic graphs and after for all cactus graphs.

\begin{theorem}
\label{Prop_cactus}Let $G\not =C_{n}$ be a cactus graph with $c$ cycles. Then
\[
\mathrm{mdim}(G)\leq L_{1}(G)+2c,
\]
and the upper bound is attained if and only if every cycle in $G$ has exactly
one vertex of degree $\geq3$.
\end{theorem}

The cyclomatic number of a graph $G$ is defined by $c(G)=m-n+1.$ As the number
of cycles in trees and graphs with edge disjoint cycles equals the cyclomatic
number, this lead the authors of \cite{SedSkrekMixed} to make the following conjecture.

\begin{conjecture}
\label{Con_inequality}Let $G\not =C_{n}$ be a graph, $c(G)$ its cyclomatic
number, and $L_{1}(G)$ the number of leaves in $G$. Then
\begin{equation}
\mathrm{mdim}(G)\leq L_{1}(G)+2c(G). \label{Inequality_conj}%
\end{equation}

\end{conjecture}

Notice that Proposition \ref{Prop_Kelm} and Theorem \ref{Prop_cactus} imply
that the equality in (\ref{Inequality_conj}) holds for all cactus graphs in
which every cycle has precisely one vertex of degree $\geq3$ (this includes
all trees and unicyclic graphs with precisely one vertex on the cycle with
degree $\geq3$). A natural question that arises is - are there any other
graphs for which the equality in (\ref{Inequality_conj}) holds? In this paper
we will try to further clarify this question.

\section{Preliminaries}

The \emph{(vertex) connectivity} $\kappa(G)$ of a graph $\emph{G}$ is the
minimum size of a vertex cut, i.e. any subset of vertices $S\subseteq V(G)$
such that $G-S$ is disconnected or has only one vertex. We say that a graph
$G$ is $k$\emph{-connected} if $\kappa(G)\geq k.$ As we are going to study the
graphs for which the equality in (\ref{Inequality_conj}) holds, it is useful
to state the following result from \cite{SedSkrekMixed}.

\begin{proposition}
\label{Prop_kapa3}Let $G$ be a $3$-connected graph. Then $\mathrm{mdim}%
(G)<2c(G).$
\end{proposition}

This proposition implies that equality in (\ref{Inequality_conj}) may hold
only for graphs with $\kappa(G)=1$ (beside cactus graphs in which every cycle
has precisely one vertex of degree $\geq3$) and $\kappa(G)=2$. A class of
graphs with $\kappa(G)=2$ which will be of interest to us are so called Theta graphs.

We say that a graph $G$ is a \emph{Theta} graph or a $\Theta$\emph{-graph} if
$G$ is a graph with two vertices $u$ and $v$ of degree $3$ and all other
vertices in $G$ are of degree $2.$ We say that a Theta graph $G$ is
\emph{balanced} if the lengths of all three paths connecting $u$ and $v$
differ by at most $1,$ otherwise we say that $G$ is \emph{unbalanced}. In this
paper we will prove that the equality in (\ref{Inequality_conj}) holds also
for balanced Theta graphs. But, before we show that, we need to introduce the
following notion which will be of use to us in the sequel.

Let $G$ be a graph and let $S\subseteq V(G)$ be a set of vertices of a graph
$G$. Any shortest path between two vertices from $S$ is called a
$S$\emph{-closed} path. Let $x$ and $x^{\prime}$ be a pair of elements from
the set $V(G)\cup E(G)$. We say that a pair $x$ and $x^{\prime}$ is
\emph{enclosed} by $S$ if there is a $S$-closed path containing $x$ and
$x^{\prime}$. We say that a pair $x$ and $x^{\prime}$ is \emph{half-enclosed}
by $S$ if there is a vertex $s\in S$ such that a shortest path from $s$ to $x$
contains $x^{\prime}$ or a shortest path from $s$ to $x^{\prime}$ contains $x$.

\begin{remark}
\label{Obs_enclosure}Let $G$ be a graph, let $S\subseteq V(G)$ be a set of
vertices in $G$ and let $x$ and $x^{\prime}$ be a pair of elements from the
set $V(G)\cup E(G)$. If $x$ and $x^{\prime}$ are enclosed by $S,$ then $x$ and
$x^{\prime}$ are distinguished by $S.$ If $x$ and $x^{\prime}$ are
half-enclosed by $S$ then $x$ and $x^{\prime}$ are distinguished by $S$ in all
cases except possibly when $x$ and $x^{\prime}$ are a pair consisting of a
vertex and an edge which are incident to each other.
\end{remark}

We say that a subgraph $H$ of a graph $G$ is an \emph{isometric} subgraph, if
for any two vertices $u,v\in V(H)$ it holds that $d_{H}(u,v)=d_{G}(u,v)$. The
following notation for paths is used. Suppose that $P$ is a path and $u,v\in
V(P),$ then by $P[u,v]$ we denote the subpath of $P$ connecting vertices $u$
and $v,$ while by $P(u,v)$ we denote $P[u,v]-\{u,v\}$. Notions $P[u,v)$ and
$P(u,v]$ are also used and they denote the subpaths where only one of the
end-vertices of $P[u,v]$ is excluded.

\section{Balanced Theta graphs}

Notice that every Theta graph $G$ has the cyclomatic number $c(G)=2.$ Also,
for every Theta graph $G$ the number of leaves equals zero, i.e. $L_{1}(G)=0.$
Therefore, for a Theta graph $G,$ the equality in (\ref{Inequality_conj}) will
hold if and only if $\mathrm{mdim}(G)=4.$ In this section we will show that
for balanced Theta graphs precisely that holds, i.e. $\mathrm{mdim}(G)=4$ if
and only if a Theta graph $G$ is balanced. First we need the following lemma.

\begin{lemma}
\label{Lemma_Theta}Let $G$ be a balanced Theta graph with vertices $u$ and $v$
of degree $3.$ Let $S\subseteq V(G)$ be a set of vertices in $G$ such that
$\left\vert S\right\vert =3$ and $S$ contains precisely one internal vertex
from each of the three distinct paths connecting vertices $u$ and $v.$ Then
$S$ is not a mixed metric generator in $G$.
\end{lemma}

\begin{proof}
Let $P_{1},$ $P_{2}$ and $P_{3}$ denote the three distinct paths in $G$
connecting vertices $u$ and $v.$ For a pair of vertices $v_{1}\in P_{i}$ and
$v_{2}\in P_{j},$ where $i\not =j,$ we denote by $d_{u}(v_{1},v_{2})$ (resp.
$d_{v}(v_{1},v_{2})$) the length of the shortest path connecting vertices
$v_{1}$ and $v_{2}$ and which contains vertex $u$ (resp. $v$). By $C_{ij}$ we
denote the cycle induced by paths $P_{i}$ and $P_{j}$. Let $S=\{s_{1}%
,s_{2},s_{3}\}$ be a set of vertices which contains precisely one internal
vertex from each of the three distinct paths connecting vertices $u$ and $v,$
where the elements of $S$ are denoted so that $s_{i}$ belongs to $P_{i}.$ Let
$P_{ij}$ and $P_{ij}^{\prime}$ be the two internally vertex disjoint paths
connecting vertices $s_{i}$ and $s_{j},$ which induce the cycle $C_{ij},$
denoted so that $\left\vert P_{ij}\right\vert \leq\left\vert P_{ij}^{\prime
}\right\vert .$ If $P_{12}$ and $P_{13}$ do not share any other vertex besides
$s_{1},$ then either $P_{12}$ and $P_{23}$ share it or $P_{13}$ and $P_{23}$
do. Therefore, at least one pair of paths $P_{ij}$ shares one more vertex
besides vertices from $S,$ say $P_{12}$ and $P_{13}.$

Next, we distinguish the following three cases.

\medskip\noindent\textbf{Case 1:} $\left\vert P_{12}\right\vert <\left\vert
P_{12}^{\prime}\right\vert $\emph{ and }$\left\vert P_{13}\right\vert
<\left\vert P_{13}^{\prime}\right\vert .$ In this case let $w$ be the neighbor
of $s_{1}$ not contained in paths $P_{12}$ and $P_{13},$ then $s_{1}$ and
$s_{1}w$ are not distinguished by $S,$ so $S$ is not a mixed metric generator.

\medskip\noindent\textbf{Case 2:} $\left\vert P_{12}\right\vert =\left\vert
P_{12}^{\prime}\right\vert $\emph{ and }$\left\vert P_{13}\right\vert
=\left\vert P_{13}^{\prime}\right\vert .$ In this case two edges incident to
$s_{1}$ are not distinguished by $S,$ so $S$ cannot be a mixed metric generator.

\medskip\noindent\textbf{Case 3:} $\left\vert P_{12}\right\vert =\left\vert
P_{12}^{\prime}\right\vert $\emph{ and }$\left\vert P_{13}\right\vert
<\left\vert P_{13}^{\prime}\right\vert .$ First notice that $\left\vert
P_{12}\right\vert =\left\vert P_{12}^{\prime}\right\vert $ implies that the
cycle $C_{12}$ is even and the pair of vertices $s_{1}$ and $s_{2}$ is an
antipodal pair on $C_{12}.$ Since $G$ is a balanced Theta graph, the fact that
$C_{12}$ is even implies $\left\vert P_{1}\right\vert =\left\vert
P_{2}\right\vert .$ Therefore, $u$ and $v$ are also an antipodal pair on
$C_{12}$ and it holds that
\begin{equation}
d(s_{1},u)=d(s_{2},v)\text{ \ \ and \ \ }d(s_{1},v)=d(s_{2}%
,u).\label{For_s1s2}%
\end{equation}
As we assumed that $P_{12}$ and $P_{13}$ share another vertex beside $s_{1},$
notice that precisely one of the vertices $u$ and $v$, say $v,$ belongs to
both $P_{12}$ and $P_{13}.$

\bigskip\noindent\textbf{Claim A.} \textit{If }$d(s_{1},v)>d(s_{2}%
,v),$\textit{ then }$\left\vert P_{23}\right\vert <\left\vert P_{23}^{\prime
}\right\vert $\textit{.}

\medskip\noindent To prove Claim A, let us assume $d(s_{1},v)>d(s_{2},v).$
Then (\ref{For_s1s2}) promptly implies $d(s_{1},u)<d(s_{2},u).$ Also, since
$P_{13}$ leads through the vertex $v,$ we have%
\[
d(s_{1},v)+d(v,s_{3})\leq d(s_{1},u)+d(u,s_{3}).
\]
Therefore, we obtain%
\begin{align*}
d_{v}(s_{2},s_{3})  &  =d(s_{2},v)+d(v,s_{3})<d(s_{1},v)+d(v,s_{3})\\
&  \leq d(s_{1},u)+d(u,s_{3})<d(s_{2},u)+d(u,s_{3})=d_{u}(s_{2},s_{3}).
\end{align*}
This means that $P_{23}$ leads through $v,$ while $P_{23}^{\prime}$ leads
through $u,$ and $\left\vert P_{23}\right\vert <\left\vert P_{23}^{\prime
}\right\vert .$ So, the claim is proven.\medskip

In the light of Claim A, notice that either $d(s_{1},v)\leq d(s_{2},v)$ or
$d(s_{1},v)>d(s_{2},v).$ If $d(s_{1},v)>d(s_{2},v),$ then by Claim A we have
$\left\vert P_{23}\right\vert <\left\vert P_{23}^{\prime}\right\vert ,$ so
switching indices $1$ and $2$ reduces this case to the case $d(s_{1},v)\leq
d(s_{2},v).$ Therefore, without loss of generality, we may assume that
$d(s_{1},v)\leq d(s_{2},v).$ From this and (\ref{For_s1s2}) we immediately
obtain
\begin{equation}
d(s_{1},v)\leq d(s_{2},v)=d(s_{1},u).\label{For_s1}%
\end{equation}
Therefore, there must exist a vertex $v^{\prime}$ on $P_{1},$ distinct from
$v,$ such that $d(v,s_{1})=d(s_{1},v^{\prime}).$ Let $a$ and $b$ be vertices
on $P_{3}$ such that $d(v,a)=d(v,s_{1})$ and $d(u,b)=d(v,s_{1}).$ This
situation is illustrated by Figure \ref{Fig_Case3} a) and we will further
consider the position of $s_{3}$ on $P_{3}.$\begin{figure}[h]
\begin{center}%
\begin{tabular}
[t]{llllll}%
a)\vspace{-0.5cm} &  & b) &  & c) & \\
& \includegraphics[scale=0.7]{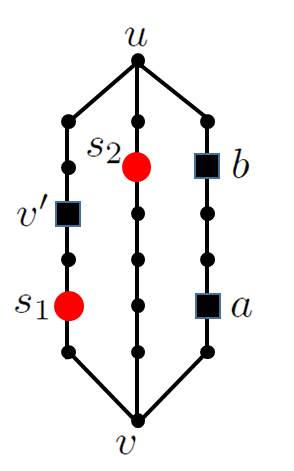} &  &
\includegraphics[scale=0.7]{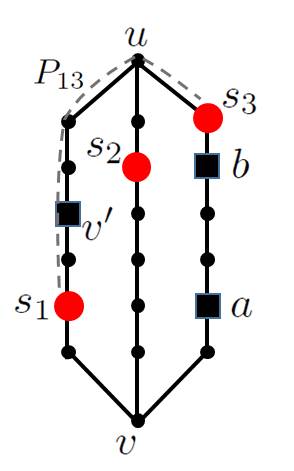} &  &
\includegraphics[scale=0.7]{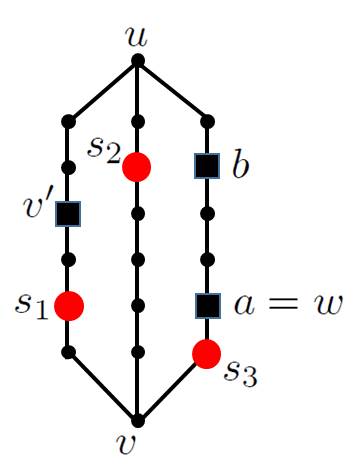}
\end{tabular}
\end{center}
\par
\caption{In the proof of Lemma \ref{Lemma_Theta}: a) the position of vertices
$s_{1},$ $s_{2},$ $v^{\prime}$, $a$ and $b$; b) the contradictory path
$P_{13}$ when $d(v,s_{3})>d(v,b)$; c) the undistioguished pair $s_{3}$ and
$s_{3}w$ when $d(v,s_{3})<d(v,a)$.}%
\label{Fig_Case3}%
\end{figure}

\bigskip\noindent\textbf{Claim B.} \textit{It holds that }$d(v,s_{3})\leq
d(v,b)$\textit{.}

\medskip\noindent Asume the contrary, i.e. $d(v,s_{3})>d(v,b)$ and notice
that
\begin{align*}
d_{u}(s_{1},s_{3})  &  <\left\vert P_{1}\right\vert -d(v,s_{1}%
)+d(u,b)=\left\vert P_{1}\right\vert ,\\
d_{v}(s_{1},s_{3})  &  =d(v,s_{1})+d(v,s_{3})>d(v,s_{1})+d(v,b)=d(v,s_{1}%
)+\left\vert P_{3}\right\vert -d(u,b)=\left\vert P_{3}\right\vert .
\end{align*}
Now, there are two possibilities, either $\left\vert P_{3}\right\vert
\geq\left\vert P_{1}\right\vert $ or $\left\vert P_{3}\right\vert <\left\vert
P_{1}\right\vert .$ Assuming $\left\vert P_{3}\right\vert \geq\left\vert
P_{1}\right\vert $ leads us to the conclusion that $d_{v}(s_{1},s_{3}%
)>\left\vert P_{3}\right\vert \geq\left\vert P_{1}\right\vert =d_{u}%
(s_{1},s_{3}),$ which means that the shortest path $P_{13}$ from $s_{1}$ to
$s_{3}$ contains $u$ which is a contradiction (see Figure \ref{Fig_Case3} b)).
Assuming the other possibility $\left\vert P_{3}\right\vert <\left\vert
P_{1}\right\vert ,$ leads us to $d_{v}(s_{1},s_{3})>\left\vert P_{3}%
\right\vert =\left\vert P_{1}\right\vert -1=d_{u}(s_{1},s_{3})-1,$ i.e.
$d_{v}(s_{1},s_{3})\geq d_{u}(s_{1},s_{3}).$ Therefore, either $d_{v}%
(s_{1},s_{3})>d_{u}(s_{1},s_{3})$ which again leads to the conclusion that
$P_{13}$ contains $u$ which is a contradiction, or $d_{v}(s_{1},s_{3}%
)=d_{u}(s_{1},s_{3})$ which implies $\left\vert P_{13}\right\vert =\left\vert
P_{13}^{\prime}\right\vert $ which contradicts the assumption of this case.
Therefore, we have proven that assumption $d(v,s_{3})>d(v,b)$ always leads to
a contradiction, so Claim B is proven.

\bigskip\noindent\textbf{Claim C.} \textit{If }$d(v,s_{3})<d(v,a),$\textit{
then }$S$\textit{ is not a mixed metric generator.}

\medskip\noindent Assume $d(v,s_{3})<d(v,a)$ and let $w$ denote the neighbor
of $s_{3}$ such that $d(w,v)>d(s_{3},v)$ as is illustrated by Figure
\ref{Fig_Case3} c). Let $x=s_{3}$ and $x^{\prime}=s_{3}w$ and notice that $x$
and $x^{\prime}$ are not distinguished by $s_{3}.$ To prove that $x$ and
$x^{\prime}$ are not distinguished by $s_{2}$ either, recall that the vertex
$a$ was chosen so that $d(v,a)=d(v,s_{1})=d(u,s_{2}).$ Having that in mind,
notice that
\begin{align*}
d_{u}(w,s_{2})  &  \geq d(a,u)+d(u,s_{2})=\left\vert P_{3}\right\vert
-d(a,v)+d(u,s_{2})=\left\vert P_{3}\right\vert ,\\
d_{v}(w,s_{2})  &  \leq d(a,v)+d(v,s_{2})=d(a,v)+\left\vert P_{2}\right\vert
-d(u,s_{2})=\left\vert P_{2}\right\vert ,
\end{align*}
so we have%
\[
d_{v}(w,s_{2})-d_{u}(w,s_{2})\leq\left\vert P_{2}\right\vert -\left\vert
P_{3}\right\vert \leq1.
\]
If $d_{v}(w,s_{2})-d_{u}(w,s_{2})\leq0$ then the shortest path from both $x$
and $x^{\prime}$ to $s_{2}$ leads through $v,$ so they are not distinguished
by $s_{2}.$ Otherwise, if $d_{v}(w,s_{2})-d_{u}(w,s_{2})=1,$ that implies
$d_{u}(w,s_{2})<d_{v}(w,s_{2})$ and $d_{v}(s_{3},s_{2})+1=d_{u}(s_{3},s_{2}),$
from which we further obtain
\begin{align*}
d(x^{\prime},s_{2})  &  =d(w,s_{2})=d_{u}(w,s_{2})\\
d(x,s_{2})  &  =d(s_{3},s_{2})=d_{v}(s_{3},s_{2})=d_{v}(w,s_{2})-1=d_{u}%
(w,s_{2})
\end{align*}
so $x$ and $x^{\prime}$ are again not distinguished by $s_{2}$. The inequality
(\ref{For_s1}) now implies that $x$ and $x^{\prime}$ are not distinguished by
$s_{1}$ either. We conclude that $x$ and $x^{\prime}$ are not distinguished by
$S,$ so the claim is established.

\medskip Given Claims B and C, the only remaining possibility is $d(v,a)\leq
d(v,s_{3})\leq d(v,b)$. Recall that $\left\vert P_{1}\right\vert =\left\vert
P_{2}\right\vert $ and let $\left\vert P_{3}\right\vert =\left\vert
P_{1}\right\vert +r$ where the fact that $G$ is balanced implies that $r$ can
take only values $-1,0,1$. Let us further denote%
\begin{equation}
q=\frac{d(v^{\prime},u)+d(u,s_{3})-d(v,s_{3})}{2}=\frac{2\left\vert
P_{1}\right\vert -2d(v,s_{1})+r-2d(v,s_{3})}{2}. \label{For_q}%
\end{equation}
Notice that the second expression for $q$ in (\ref{For_q}) implies that $q$ is
integer if and only if $r=0.$ We also want to use the first expression for $q$
in (\ref{For_q}) to derive a bound on $d(v^{\prime},u)$ from it. For that
purpose notice that
\begin{align*}
d(u,s_{3})-d(v,s_{3})  &  =\left\vert P_{3}\right\vert -2d(s_{3},v)=\left\vert
P_{1}\right\vert +r-2d(v,a)-2d(s_{3},a)\\
&  \leq\left\vert P_{1}\right\vert +r-2d(v,a)=\left\vert P_{1}\right\vert
+r-2d(v,s_{1})=d(v^{\prime},u)+r
\end{align*}
with equality holding if and only if $d(s_{3},a)=0.$ Plugging this in
(\ref{For_q}) yields
\[
q\leq\frac{d(v^{\prime},u)+d(v^{\prime},u)+r}{2}%
\]
which further implies $d(v^{\prime},u)\geq q-\frac{r}{2}.$ Let $w$ be the
vertex from $P_{1}[v^{\prime},u]$ such that $d(v^{\prime},w)=\left\lfloor
q\right\rfloor .$ Notice that such a vertex $w$ must exist on $P_{1}$ because
the fact that $q$ is integer only for $r=0$ implies
\[
d(v^{\prime},w)=\left\lfloor q\right\rfloor \leq q-\frac{r}{2}\leq
d(v^{\prime},u).
\]
Moreover, notice that in the case when $r=-1$ even stricter upper bound for
$d(v^{\prime},w)$ holds, i.e. $d(v^{\prime},w)\leq q-\frac{r}{2}-1\leq
d(v^{\prime},u)-1$. Now, let $w^{\prime}$ be the vertex on $P_{2}$ such that
$d(w^{\prime},v)=\left\lfloor q\right\rfloor $.

\bigskip\noindent\textbf{Claim D.} \textit{It holds that }$d_{u}%
(w,s_{3})<d_{v}(w,s_{3}).$

\medskip\noindent Notice that%
\begin{align*}
d_{u}(w,s_{3}) &  =d(w,u)+d(u,s_{3})=d(v^{\prime},u)-d(w,v^{\prime}%
)+d(u,s_{3}),\\
d_{v}(w,s_{3}) &  =d(w,v^{\prime})+2(v,s_{1})+d(v,s_{3}).
\end{align*}
Therefore, we have%
\[
d_{u}(w,s_{3})-d_{v}(w,s_{3})=d(v^{\prime},u)-2d(w,v^{\prime})+d(u,s_{3}%
)-2(v,s_{1})-d(v,s_{3}).
\]
Now the fact that $d(v^{\prime},w)=\left\lfloor q\right\rfloor $ and the
definition of $q$ further imply%
\[
d_{u}(w,s_{3})-d_{v}(w,s_{3})=-2\left\lfloor q\right\rfloor -2(v,s_{1})+2q<0
\]
which concludes the proof of Claim D.

\bigskip\noindent\textbf{Claim E.} \textit{It holds that }$d_{v}(w^{\prime
},s_{3})<d_{u}(w^{\prime},s_{3}).$

\medskip\noindent Notice that
\begin{align*}
d_{v}(w^{\prime},s_{3}) &  =d(w^{\prime},v)+d(v,s_{3})=\left\lfloor
q\right\rfloor +d(v,s_{3}),\\
d_{u}(w^{\prime},s_{3}) &  =d(w^{\prime},u)+d(u,s_{3})=\left\vert
P_{2}\right\vert -\left\lfloor q\right\rfloor +d(u,s_{3}).
\end{align*}
Therefore, having in mind that $\left\vert P_{2}\right\vert =\left\vert
P_{1}\right\vert $ we further have%
\begin{align*}
d_{v}(w^{\prime},s_{3})-d_{u}(w^{\prime},s_{3}) &  =2\left\lfloor
q\right\rfloor +d(v,s_{3})-\left\vert P_{1}\right\vert -d(u,s_{3})=\\
&  =2\left\lfloor q\right\rfloor +d(v,s_{3})-d(v^{\prime},u)-2d(v,s_{1}%
)-d(u,s_{3})
\end{align*}
from which, given the definition of $q,$ we obtain%
\[
d_{v}(w^{\prime},s_{3})-d_{u}(w^{\prime},s_{3})=2\left\lfloor q\right\rfloor
-2d(v,s_{1})-2q<0
\]
which proves Claim E.\medskip

Now we distinguish the following three subcases with respect to the value of
$r.$

\medskip\noindent\textbf{Subcase 3.a:} $r=-1.$ Recall that in this case
$d(v^{\prime},w)\leq d(v^{\prime},u)-1,$ which implies $w\not =u,$ so there
exists a neighbor $z$ of the vertex $w$ on $P_{1},$ which is further from
$s_{1}$ than $w.$ Also, let $z^{\prime}$ be the neighbor of $w^{\prime}$ on
$P_{2}$ which is further from $v$ than $w^{\prime}.$ We want to prove that the
edges $x=wz$ and $x^{\prime}=w^{\prime}z^{\prime}$ are not distinguished by
$S$, see Figure \ref{Fig_r} a) for illustration. First note that
\[
d(x,s_{1})-d(x^{\prime},s_{1})=\left\lfloor q\right\rfloor +d(v^{\prime}%
,s_{1})-(\left\lfloor q\right\rfloor +d(v,s_{1}))=0
\]
which implies that $s_{1}$ does not distinguish $x$ and $x^{\prime}.$ Since
$s_{1}$ and $s_{2}$ are an antipodal pair of vertices on the even cycle
$C_{12}$, if $s_{1}$ does not distinguish $x$ and $x^{\prime}$ then $s_{2}$
does not distinguish them either. The only remaining possibility is for
$s_{3}$ to distinguish $x$ and $x^{\prime},$ but Claims D and E here imply%
\[
d(x,s_{3})-d(x^{\prime},s_{3})=d(v^{\prime},u)-\left\lfloor q\right\rfloor
-1+d(u,s_{3})-(\left\lfloor q\right\rfloor +d(v,s_{3})),
\]
where the definition of $q$ further implies%
\[
d(x,s_{3})-d(x^{\prime},s_{3})=-2\left\lfloor q\right\rfloor -1+2q=0
\]
from which we conclude that $s_{3}$ does not distinguish $x$ and $x^{\prime}$
either. Therefore, $x$ and $x^{\prime}$ are not distinguished by $S,$ which
means that $S$ cannot be a mixed metric generator.

\begin{figure}[h]
\begin{center}%
\begin{tabular}
[t]{llllll}%
a)\vspace{-0.5cm} &  & b) &  & c) & \\
& \includegraphics[scale=0.7]{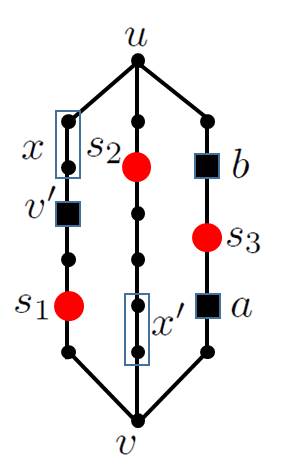} &  &
\includegraphics[scale=0.7]{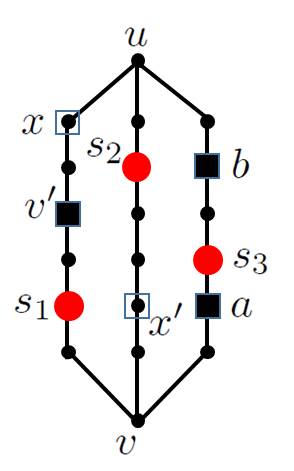} &  &
\includegraphics[scale=0.7]{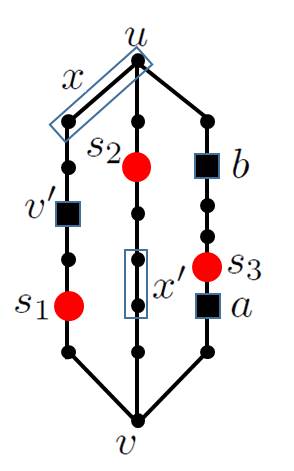}
\end{tabular}
\end{center}
\par
\caption{With the proof of Lemma \ref{Lemma_Theta}, the undistinguished pair
$x$ and $x^{\prime}$ when $d(v,a)\leq d(v,s_{3})\leq d(v,b)$\ and a) $r=-1$;
b) $r=0$; c) $r=1$.}%
\label{Fig_r}%
\end{figure}

\medskip\noindent\textbf{Subcase 3.b:} $r=0.$ In this subcase vertices $x=w$
and $x^{\prime}=w^{\prime}$ are not distinguished by $S$ (see Figure
\ref{Fig_r} b)), which is shown similarly as in Subcase 2.a, so $S$ cannot be
a mixed metric generator.

\medskip\noindent\textbf{Subcase 3.c:} $r=1.$ In this subcase notice that
$w=u$ if and only if $d(s_{3},a)=0,$ i.e. $s_{3}=a.$ When $s_{3}=a$ then
$x=s_{3}$ and $x^{\prime}=s_{3}z$ are not distinguished by $S,$ so we may
assume $s_{3}\not =a$ which implies $w\not =u,$ so there is a neighbor $z$ of
$w$ on $P_{1}$ further from $s_{1}$ than $w$. Now, let $z^{\prime}$ be the
neighbor of $w^{\prime}$ on $P_{2}$ which is further from $v$ than $w^{\prime
}.$ Then $x=wz$ and $x^{\prime}=w^{\prime}z^{\prime}$ are not distinguished by
$S$ (see Figure \ref{Fig_r} c)), which is shown by a similar calculation as in
Subcase 3.a.
\end{proof}

We will use Lemma \ref{Lemma_Theta} to prove the exact value of the mixed
metric dimension of balanced Theta graphs.

\begin{theorem}
\label{Prop_Theta}If $G$ is a balanced Theta graph, then $\mathrm{mdim}(G)=4.$
\end{theorem}

\begin{proof}
Let $u$ and $v$ be the two vertices in $G$ of degree $3$ and let $w$ and $z$
be two neighbors of $v.$ We want to prove that $S=\{u,v,w,z\}$ is a mixed
metric generator. Let $x$ and $x^{\prime}$ be a pair of elements from
$V(G)\cup E(G).$ Let $P_{1},$ $P_{2}$ and $P_{3}$ be the three distinct paths
connecting vertices $u$ and $v$ in $G.$ If $x$ and $x^{\prime}$ belong to a
same path $P_{i}$ then they are certainly distinguished by $u\in S$ or by
$v\in S$. Assume therefore that $x$ belongs to $P_{i}$ and $x^{\prime}$
belongs to $P_{j}$ where $i\not =j$. If $x$ and $x^{\prime}$ are distinguished
by $u$ or $v,$ then the proof is done, so we may assume that they are not
distinguished by either $u$ or $v$. If $C_{ij}$ is even, this implies $x$ and
$x^{\prime}$ are a pair of vertices or a pair of edges, while in the case of
odd cycle $C_{ij}$ the pair $x$ and $x^{\prime}$ must be a mixed pair
consisting of a vertex and an edge. Notice that at least one of the paths
$P_{i}$ and $P_{j}$ must contain at least one more vertex from $S$ besides $u$
and $v$, say $P_{i}$ contains $w.$ Then $w$ certainly distinguishes $x$ and
$x^{\prime},$ so we proved that $S$ is a mixed metric generator in $G$ which
implies $\mathrm{mdim}(G)\leq4.$

To conclude the proof, we still need to prove that any set $S$ with
$\left\vert S\right\vert <4$ cannot be a mixed metric generator. Asume first
that both $u$ and $v$ are contained in $S,$ then $\left\vert S\right\vert <4$
implies that two of the paths $P_{1},$ $P_{2}$ and $P_{3}$ do not share an
internal vertex with $S,$ say $P_{i}$ and $P_{j}.$ Let $u_{i}$ and $u_{j}$ be
the neigbors of vertex $u$ on paths $P_{i}$ and $P_{j}$ respectively. If
$P_{i}$ and $P_{j}$ are of the same length, then $u_{i}$ and $u_{j}$ are not
distinguished by $S,$ so $S$ cannot be a mixed metric generator. If $P_{i}$
and $P_{j}$ are not of the same length, say $\left\vert P_{i}\right\vert
<\left\vert P_{j}\right\vert ,$ then the fact that $G$ is balanced implies
$\left\vert P_{i}\right\vert =\left\vert P_{j}\right\vert -1$, but this
further implies $u$ and $uu_{j}$ are not distinguished by $S.$

Assume now that precisely one of the vertices $u$ and $v$ is contained in $S,$
say $u$. This implies that there is a path $P_{i}$ in $G$ which does not share
any other vertex with $S$ besides $u$. Denote by $u_{i}$ the neighbor of $u$
on $P_{i}$. Since $G$ is a balanced Theta graph it holds that $d(u_{i},w)\geq
d(u,w)$ for every internal vertex $w$ from the path $P_{j},$ $j\not =i$.
Therefore, $u$ and $uu_{i}$ are not distinguished by $S$.

Finaly, assume that neither $u$ nor $v$ are contained in $S.$ If there is a
path $P_{i}$ which does not share an internal vertex with $S,$ then $u$ and
$uu_{i}$ are obviously not distinguished by $S.$ Assume therefore that each
path $P_{i}$ shares at least one internal vertex with $S,$ which together with
the fact $\left\vert S\right\vert <4$ further implies $\left\vert S\right\vert
=3$ and each $P_{i}$ shares precisely one internal vertex with $S$. But then
Lemma \ref{Lemma_Theta} implies $S$ cannot be a mixed metric generator and the
proof is established.
\end{proof}

Since $c(G)=2$ and $L_{1}(G)=0$ holds for any Theta graph $G,$ Theorem
\ref{Prop_Theta} immediately yields the following corollary.

\begin{corollary}
For a balanced Theta graph $G,$ it holds that $\mathrm{mdim}(G)=L_{1}%
(G)+2c(G).$
\end{corollary}

This result implies that Conjecture \ref{Con_inequality} holds for balanced
Theta graphs, moreover it holds with equality in (\ref{Inequality_conj}).

\section{Unbalanced Theta graphs}

To complete the results we will now prove that Conjecture \ref{Con_inequality}
holds also for unbalanced Theta graphs, but for them the equality in
(\ref{Inequality_conj}) does not hold.

\begin{lemma}
\label{Lemma_unbalanced}Let $G$ be an unbalanced Theta graph, then
$\mathrm{mdim}(G)\leq3.$
\end{lemma}

\begin{proof}
Let $u$ and $v$ be the two vertices of degree $3$ in $G$ and let $P_{1},$
$P_{2}$ and $P_{3}$ be three distinct paths in $G$ connecting vertices $u$ and
$v$, where without loss of generality we may assume that $\left\vert
P_{1}\right\vert \leq\left\vert P_{2}\right\vert \leq\left\vert P_{3}%
\right\vert .$ Since $G$ is an unbalanced Theta graph, it follows that
$\left\vert P_{3}\right\vert -\left\vert P_{1}\right\vert \geq2.$ By $C_{ij}$
we denote the cycle induced by paths $P_{i}$ and $P_{j}$. Let us now consider
the cycle $C_{13}$ and vertices $u$ and $v$ which belong to it. If the cycle
$C_{13}$ is even, then each of the two vertices $u$ and $v$ has precisely one
antipodal vertex on $C_{13}$. On the other hand, if $C_{13}$ is an odd cycle,
then vertices $u$ and $v$ each have precisely two antipodal vertices.

Let us denote antipodal vertices of $u$ on $C_{13}$ by $a_{u}$ and
$a_{u}^{\prime}$ (where we assume $a_{u}=a_{u}^{\prime}$ when $C_{13}$ is an
even cycle) and by $a_{v}$ and $a_{v}^{\prime}$ the two antipodals of the
vertex $v$. We may asume that the pair of antipodal vertices $a_{u}$ and
$a_{u}^{\prime}$ are denoted so that $d(v,a_{u})\leq d(v,a_{u}^{\prime}).$
Similarly, we denote vertices $a_{v}$ and $a_{v}^{\prime}$ so that the
inequality $d(u,a_{v})\leq d(u,a_{v}^{\prime})$ holds. From $\left\vert
P_{3}\right\vert -\left\vert P_{1}\right\vert \geq2$ it follows that all of
the vertices $a_{u},$ $a_{u}^{\prime},$ $a_{v},$ $a_{v}^{\prime}$ belong to
$P_{3},$ where $a_{u}$ and $a_{u}^{\prime}$ are distinct from $v$ and,
similarly $a_{v}$ and $a_{v}^{\prime}$ are distinct from $u.$ Further, let $w$
be a vertex from the path $P_{2}$ such that distance from $w$ to vertices $u$
and $v$ differs by at most one (i.e. $w$ is the middle or "almost middle"
vertex of the path $P_{2}$). Finally, we define the set $S=\{a_{u},a_{v},w\}$
for which we will prove that it is a mixed metric generator in $G.$ All these
vertices and the set $S$ are illustrated by Figure \ref{Fig_unbalanced}.
\begin{figure}[h]
\begin{center}%
\begin{tabular}
[t]{llll}%
a)\vspace{-0.5cm} &  & b) & \\
& \includegraphics[scale=0.7]{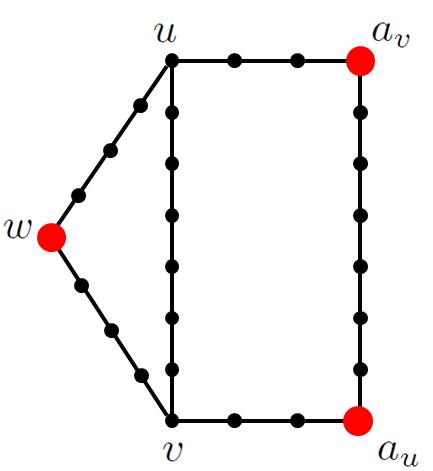} &  &
\includegraphics[scale=0.7]{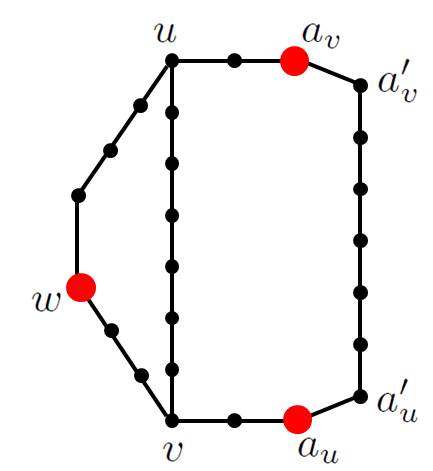}
\end{tabular}
\end{center}
\par
\caption{In the proof of Lemma \ref{Lemma_unbalanced}, the position of
vertices $w,$ $a_{u},$ $a_{u}^{\prime},$ $a_{v},$ $a_{v}^{\prime}$ and the
mixed metric generator $S=\{w,a_{u},a_{v}\}$ in the case when both $P_{2}$ and
$C_{13}$ are: a) of even length, b) of odd length.}%
\label{Fig_unbalanced}%
\end{figure}

In order to prove that $S$ is a mixed metric generator, let $x$ and
$x^{\prime}$ be a pair of elements from the set $V(G)\cup E(G).$ We
distinguish the following three cases regarding the position of $x$ and
$x^{\prime}$ in $G.$

\medskip\noindent\textbf{Case 1:} \emph{Both }$x$\emph{ and }$x^{\prime}%
$\emph{ belong to }$C_{13}.$ Notice that there are two subpaths of $C_{13}$
which connect vertices $a_{u}$ and $a_{v},$ one of them is $P_{3}[a_{u}%
,a_{v}],$ the other is $C_{13}-P_{3}(a_{u},a_{v}).$ If $x$ and $x^{\prime}$
belong to different subpaths of $C_{13}$ connecting $a_{u}$ and $a_{v},$ since
$a_{u}$ and $a_{v}$ is not an antipodal pair on $C_{13},$ it follows that $x$
and $x^{\prime}$ are distinguished by $a_{u}$ or $a_{v}.$ If both $x$ and
$x^{\prime}$ belong to $P_{3}[a_{u},a_{v}],$ they are distinguished by $a_{u}$
or $a_{v}$ according to Observation \ref{Obs_enclosure}. The only remaining
possibility is that both $x$ and $x^{\prime}$ belong to $C_{13}-P_{3}%
(a_{u},a_{v}).$ Here we distinguish several further possibilities. If $x$ or
$x^{\prime}$ is an internal vertex or an edge from $P_{1},$ then $x$ and
$x^{\prime}$ are distinguished by $a_{u}$ or $a_{v}.$ Similarly, if one from
the pair $x$ and $x^{\prime}$ belongs to $P_{3}[u,a_{v}]$ and the other to
$P_{3}[v,a_{u}]$, then again $x$ and $x^{\prime}$ are distinguished by $a_{u}$
or $a_{v}.$ Now, if both $x$ and $x^{\prime}$ belong to $P_{3}[u,a_{v}]$ they
are enclosed by $a_{v}$ and $w,$ so they are distinguished by $S$ according to
Observation \ref{Obs_enclosure}. Finally, if both $x$ and $x^{\prime}$ belong
to $P_{3}[v,a_{u}]$, they are enclosed by $a_{u}$ and $w,$ so Observation
\ref{Obs_enclosure} again implies $x$ and $x^{\prime}$ are distinguished by
$S$.

\medskip\noindent\textbf{Case 2:} \emph{Both }$x$\emph{ and }$x^{\prime}%
$\emph{ belong to }$P_{2}$\emph{.} Notice that if both $x$ and $x^{\prime}$
belong to $P_{2}[u,w],$ or they both belong to $P_{2}[v,w],$ then $x$ and
$x^{\prime}$ are enclosed by $a_{v}$ and $w$ in the first case and by $a_{u}$
and $w$ in the second case, either way they are $S$-enclosed and therefore
distinguished by $S$ according to Observation \ref{Obs_enclosure}. The only
remaining possibility is when they belong to different sides of $w,$ say $x$
belongs to $P_{2}[u,w]$ and $x^{\prime}$ belongs to $P_{2}[v,w],$ with both
$x$ and $x^{\prime}$ being distinct from $w$. In this case if $x$ and
$x^{\prime}$ are not distinguished by $w,$ this implies that $\left\vert
d(x,u)-d(x^{\prime},v)\right\vert \leq2$ where without loss of generality we
may assume $d(x,u)\leq d(x^{\prime},v)$, i.e. $0\leq d(x^{\prime
},v)-d(x,u)=\Delta$ where $\Delta\leq2$. If the shortest path from $x$ to
$a_{u}$ leads through $w,$ then $x$ and $x^{\prime}$ are distinguished by
$a_{u}$ according to Observation \ref{Obs_enclosure}. The similar argument
holds when the shortest path from $x^{\prime}$ to $a_{v}$ leads through $w.$
So, let us assume the opposite, i.e. that the shortest path from $x$ to
$a_{u}$ and from $x^{\prime}$ to $a_{v}$ leads through $P_{1}$. Notice that if
$C_{13}$ is even, then there is a shortest path which leads both through
$P_{1}$ and another one through $P_{3}.$ And if $C_{13}$ is odd, then there is
only one which leads only through $P_{1}$. Hence, we have
\begin{align*}
d(x,a_{u}) &  =d(x,u)+\left\vert P_{1}\right\vert +d(v,a_{u})=d(x^{\prime
},v)-\Delta+\left\vert P_{1}\right\vert +d(v,a_{u})=\\
&  =d(x^{\prime},a_{u})-\Delta+\left\vert P_{1}\right\vert .
\end{align*}
Therefore, $x$ and $x^{\prime}$ are not distinguished by $a_{u}$ only when
$\left\vert P_{1}\right\vert =\Delta.$ But in that case we have
\begin{align*}
d(x^{\prime},a_{v}) &  =d(x^{\prime},v)+\left\vert P_{1}\right\vert
+d(u,a_{v})=d(x,u)+\Delta+\left\vert P_{1}\right\vert +d(u,a_{v})=\\
&  =d(x,a_{v})+2\left\vert P_{1}\right\vert >d(x,a_{v}),
\end{align*}
so $x$ and $x^{\prime}$ are distinguished by $a_{v}\in S$ and we are finished.

\medskip\noindent\textbf{Case 3:} $x$\emph{ belongs to }$P_{2}$\emph{ and
}$x^{\prime}$\emph{ belongs to }$C_{13}.$ Recall that $w$ is the middle (or
"almost" middle) vertex of the path $P_{2}$, where without the loss of
generality we may assume $d(w,u)\geq d(w,v).$ Denote by $u_{i}$ the neighbor
of $u$ on the path $P_{i}$ and by $v_{i}$ the neighbor of $v$ on the path
$P_{i}.$ Now, when $P_{2}$ is of even length, then $x$ and $x^{\prime}$ are
not distinguished by $w$ only if $x\in\{u,v\}$ and $x^{\prime}\in
\{u,v,uu_{1},uu_{3},vv_{1},vv_{3}\}.$ But, since both $u$ and $v$ belong also
to $C_{13}$, this means $x$ belongs also to $C_{13},$ so this case reduces to
Case 1. Assume, therefore, that the length of $P_{2}$ is odd. Recall that in
this case $d(w,u)>d(w,v).$ Notice that $x$ and $x^{\prime}$ are not
distinguished by $w$ only if $x\in\{u_{2},u_{2}u\}$ and $x^{\prime}%
\in\{v,vv_{1},vv_{3}\}.$ Also, notice that there certainly exists a shortest
path from $x$ to $a_{u}$ which leads through $v$ (if there is a shortest path
from $x$ to $a_{u}$ which leads through $P_{3},$ then $C_{13}$ is even and
there is also a shortest path which leads through $P_{1},$ so the claim
holds). Finally, notice that $d(x,a_{u})>d(v,a_{u}).$ On the other hand it
obviously holds that $d(x^{\prime},a_{u})\leq d(v,a_{u}).$ We conclude that
$d(x,a_{u})>d(x^{\prime},a_{u}),$ which implies that $x$ and $x^{\prime}$ are
distinguished by $S.$

\medskip We have established that any $x$ and $x^{\prime}$ are distinguished
by $S,$ so $S$ is a mixed metric generator. Since $\left\vert S\right\vert
=3,$ this implies $\mathrm{mdim}(G)\leq3$.
\end{proof}

\begin{theorem}
If $G$ is an unbalanced Theta graph, then $\mathrm{mdim}(G)=3.$
\end{theorem}

\begin{proof}
Given the result from Lemma \ref{Lemma_unbalanced}, it is sufficient to prove
that a set $S\subseteq V(G)$ such that $\left\vert S\right\vert =2$ cannot be
a mixed metric generator. We use the same notation as before, i.e. $P_{1}$,
$P_{2}$ and $P_{3}$ are the three paths connecting vertices $u$ and $v$ of
degree $3$ in $G$ denoted so that $\left\vert P_{1}\right\vert \leq\left\vert
P_{2}\right\vert \leq\left\vert P_{3}\right\vert .$ By $C_{ij}$ we denote the
cycle induced by paths $P_{i}$ and $P_{j}$. Let us further denote by $s_{1}$
and $s_{2}$ the only pair of elements from $S.$

Assume first that both $s_{1}$ and $s_{2}$ belong to $C_{12}$. As the mixed
metric dimension of any cycle equals three, it follows that $S$ is not a mixed
metric generator in $C_{12}.$ Since $C_{12}$ is an isometric subgraph of $G$
it follows that $S$ cannot be a mixed metric generator in $G$ either. The
similar argument holds when both $s_{1}$ and $s_{2}$ belong to $C_{13}$. The
only remaining possibility is that $s_{1}$ belongs to $P_{2}$ and $s_{2}$
belongs to $P_{3}$. Notice that both $s_{1}$ and $s_{2}$ in this case must be
of degree two, otherwise the pair $s_{1}$ and $s_{2}$ would belong to $C_{12}$
or $C_{13}$ and the case would reduce to the already proven cases. Let $w$ and
$z$ be two neighbors of $s_{1}$. There are only two possibilities with respect
to the distances from $w$ and $z$ to $s_{2}$, it is either $d(w,s_{2}%
)=d(z,s_{2})$ or $d(w,s_{2})\not =d(z,s_{2})$. If $d(w,s_{2})=d(z,s_{2})$,
then $ws_{2}$ and $zs_{2}$ are not distinguished by $S$. On the other hand, if
$d(w,s_{2})\not =d(z,s_{2})$ we may, without the loss of generality, assume
that $d(w,s_{2})<d(z,s_{2})$. But then $s_{2}$ and $s_{2}z$ are not
distinguished by $S$ and we are finished.
\end{proof}

As for an unbalanced Theta graph $G$ we have $L_{1}(G)=0$ and $c(G)=2,$ we
immediately obtain the following result.

\begin{corollary}
For an unbalanced Theta graph $G$ it holds that $\mathrm{mdim}(G)<L_{1}%
(G)+2c(G).$
\end{corollary}

We conclude from this result that Conjecture \ref{Con_inequality} holds for
the class of unbalanced Theta graphs with the strict inequality in
(\ref{Inequality_conj}).

\section{Concluding remarks}

In \cite{SedSkrekMixed} it was conjectured that $\mathrm{mdim}(G)\leq
L_{1}(G)+2c(G)$ for all graphs, where $c(G)$ is the cyclomatic number and
$L_{1}(G)$ the number of leaves in a graph $G$ (see Conjecture
\ref{Con_inequality}). In this paper we focused our interest on graphs for
which the conjecture holds with equality. It was already proven in literature
that the equality holds for all trees, even more for all cactus graphs in
which every cycle has precisely one vertex of degree $\geq3$. We wanted to
find other graphs for which the equality holds. By Proposition
\ref{Prop_kapa3} we know that the equality can hold only for graphs with
$\kappa(G)=1$ or $\kappa(G)=2$. Since all cactus graphs (except a cycle graph)
have vertex conectivity equal to $1,$ this means there were no known graphs
with $\kappa(G)=2$ for which the equality holds. In this paper we found such a
family, i.e. balanced Theta graphs, for which the equality also holds. We
further proved that for unbalanced Theta graphs Conjecture
\ref{Con_inequality} also holds, but with strict inequality in
(\ref{Inequality_conj}). The natural question that arises is: \textit{Are
there any other graphs for which the equality holds?} Our investigation of the
question leads us to the opinion that there are not, i.e. having in mind that
trees are a subclass of cactus graphs we state the following formal conjecture.

\begin{conjecture}
\label{Con_Attained}For a graph $G$ it holds that $\mathrm{mdim}%
(G)=L_{1}(G)+2c(G)$ if and only if $G$ is a cactus graph in which every cycle
has precisely one vertex of degree $\geq3$ or a balanced Theta graph.
\end{conjecture}

\bigskip\noindent\textbf{Acknowledgements.}~~The authors acknowledge partial
support Slovenian research agency ARRS program \ P1--0383 and ARRS project
J1-1692 and also Project KK.01.1.1.02.0027, a project co-financed by the
Croatian Government and the European Union through the European Regional
Development Fund - the Competitiveness and Cohesion Operational Programme.

\end{document}